\documentclass[12pt, letterpaper]{amsart}

\usepackage{amssymb}
\usepackage{amsthm}
\usepackage{amsxtra}
\usepackage{graphicx}
\usepackage{listings}
\usepackage{color}
\usepackage{MnSymbol}
\usepackage{enumitem}
\usepackage{booktabs}

\usepackage[utf8]{inputenc}
\usepackage[T1,T2A]{fontenc}

\usepackage[numbers]{natbib}

\newcommand{\cR}{\mathbb{R}}
\newcommand{\la}{\langle}
\newcommand{\ra}{\rangle}
\newcommand{\lla}{\llangle}
\newcommand{\rra}{\rrangle}

\newcommand{\sech}{\operatorname{sech}}
\newcommand{\defeq}{\stackrel{\rm{def}}{=}}

\newcommand{\spn}{\operatorname{span}}

\renewcommand{\arraystretch}{1.3}
\usepackage[margin=0.8in]{geometry}
\usepackage{adjustbox}
\usepackage{multirow}

\newtheorem{theorem}{Theorem}

\newtheorem{lemma}[theorem]{Lemma}
\newtheorem{corollary}[theorem]{Corollary}

\theoremstyle{remark}
\newtheorem{remark}[theorem]{Remark}

\numberwithin{equation}{section}
\numberwithin{table}{section}
\numberwithin{figure}{section}
\numberwithin{theorem}{section}

\author[J. Holmer]{Justin Holmer}
\address{Department of Mathematics\\Brown University\\Providence, RI 02912, USA}
\curraddr{} 
\email{justin\_holmer@brown.edu}
\thanks{}

\author[S. Roudenko]{Svetlana Roudenko}
\address{Department of Mathematics \& Statistics\\Florida International University,  Miami, FL 33199, USA}
\curraddr{}
\email{sroudenko@fiu.edu}
\thanks{}

\date{}



\title[Spectral property for 2D ZK]{Spectral property\\ for the 2D Zakharov-Kuznetsov equation}

\begin{document}

\begin{abstract}
We discuss a spectral property for the virial operator of the 2D Zakharov-Kuznetsov (ZK) equation. This is a crucial ingredient to establish blow-up or asymptotic stability of solitary waves in higher-dimensional problems. 

This model in 3D setting was originally introduced by Zakharov and Kuznetsov in plasma physics, and is also a higher-dimensional generalization of the well-known Korteweg-de Vries (KdV) equation. The problem of stability of solitary waves in ZK equation or stable blow-up in modified ZK (or KdV-type) equation is an important physical question, for which virial operators and their spectral properties are the essential elements of the analysis. In this paper we investigate this problem analytically and reduce it to verifying numerically only some signs of inner products and certain eigenvalues.
\end{abstract}

\maketitle


\section{Introduction}

One of the fundamental models in physics is the equation of  
Zakharov and Kuznetsov that they introduced in 70s \cite{ZK1972, ZK},
\begin{equation}\label{ZK}
\partial_t u + \partial_x \left(\Delta_{(x,y,z)} u  + u^2 \right) = 0.  
\end{equation}
This equation is a higher-dimensional extension of the well-known and extensively studied Korteweg-de Vries (KdV) equation, which models, for example, the weakly nonlinear waves in shallow water:
\begin{equation}
\label{gKdV}
\qquad \partial_t u + \partial_x(\partial_x^2 u + u^p) = 0, \quad p=2, \qquad x \in \cR, \quad t \in \cR. 
\end{equation}
The KdV equation is completely integrable, together with its modified version for $p=3$, called mKdV. When other integer powers $p \neq 2, 3$ are considered, it is referred to as the {\it generalized} KdV (gKdV) equation (and is non-integrable). 

In this paper we are interested in the generalized version of Zakharov-Kuznetsov (gZK) equation in two dimensions, namely,  
\begin{equation}\label{gZK}
\qquad 
\partial_t u + \partial_x \left(\Delta_{(x,y)} u  +  |u|^{p-1} \,  u \right) = 0, \quad (x,y) \in \mathbb{R}^2, \quad t\in \mathbb{R},
\end{equation}
with $u$ real-valued and $\Delta_{(x,y)} = \partial_x^2 +\partial_{y}^2$, which for brevity we also denote by $\Delta$. The absolute value sign in the nonlinearity allows one to consider non-integer powers, otherwise, the absolute sign can be omitted. We are specifically interested in gZK equation with the power of nonlinearity $p=3$, since this is a critical equation, where existence of blow-up was expected \cite{K} and numerical simulations are shown in \cite{KRS}; it was proved in our previous work \cite{FHRY}, however, relied on numerically computed certain spectral properties. Besides this critical power, we also consider other (subcritical) powers $1<p<3$, where the solutions should be global and asymptotic stability of solitary waves is expected (in the case when $p$ is close to $2$, see  \cite{CMPS}). We mention that in 3D, the question of Zakharov and Kuznetsov about stability of solitary waves \cite{ZK1972, ZK} in the ZK equation  \eqref{ZK} was recently resolved in our paper \cite{FHRY}, showing the asymptotic stability of solitary waves (also relying on numerically computed spectral properties of a virial operator), for numerical simulations of solutions in the three-dimensional case refer to \cite{KRS3}.

One of the crucial parts of establishing blow-up (in the critical case) or asymptotic stability of solitary waves (in subcritical cases) for the gZK (as a higher-dimensional model) or one-dimensional KdV-type equations is to establish a {\it coercivity} property of a certain virial-type operator, which relies on understanding its spectral properties. 
This is exactly what we investigate in this paper.

During their lifespan, the solutions $u(x,y,t)$ to \eqref{gZK} conserve the $L^2$ norm (often referred to as momentum or mass) and energy (Hamiltonian):
\begin{equation}\label{MC}
M[u(t)]=\int_{\cR^2} \left[u(x,y,t) \right]^2 \, dx\, dy  = M[u(0)],
\end{equation}
\begin{equation}\label{EC}
E[u(t)]=\dfrac{1}{2}\int_{\cR^2}|\nabla u(x,y,t)|^2 \, dx\, dy  - \dfrac{1}{p+1}\int_{\cR^2} \left|u(x,y,t) \right|^{p+1} \, dx\, dy  = E[u(0)].
\end{equation}
(The integral of the solution is also preserved in time, though we omit it as it is not needed in this paper, see \cite[(1.6)]{FHRY2}.) We also mention that unlike the KdV and mKdV, which are completely integrable, the gZK equations do not exhibit complete integrability for any $p$, hence, there are only 3 conserved quantities. 

The {\it scaling invariance} (an appropriately rescaled version of the original solution is also a solution of the equation) is given by 
$$
u_\lambda(x,y,t)=\lambda^{\frac{2}{p-1}} u(\lambda x, \lambda y, \lambda^3t),
$$
which makes the Sobolev $\dot{H}^s$-norm invariant with 
$$
s = 1 - \frac2{p-1}.
$$
When $s=0$ (this corresponds to $p=3$), and thus, the equation \eqref{gZK} is called the $L^2$-critical equation; when $1<p<3$ (i.e., $s<0$) the gZK equation is subcritical. This is the range of $p$ ($1<p \leq 3$) that we consider in this paper.

The generalized Zakharov-Kuznetsov equation has a family of {\it traveling solitary waves} that propagate only in positive $x$-direction,
\begin{equation}\label{Eq:TW}
u(x,y,t) = Q_c(x-ct, y), \, c>0,
\end{equation}
with $Q_c \to 0$ as $|(x,y)| \to \infty$. Here, $Q_c$ is the rescaled ground state $Q$, namely, 
$$
Q_c(x,y) = c^{\frac{2}{p-1}} Q(cx, cy)
$$
with $Q$ being a positive radial solution $H^1(\cR^2)$ of the nonlinear elliptic equation 
\begin{equation}\label{E:Q}
-Q+\Delta Q +|Q|^{p-1} Q = 0.
\end{equation}
Uniqueness of the ground state was established in \cite{K89}, for other properties, see for instance \cite{FHR3}. 
Note that $Q \in C^{\infty}(\cR^2)$, $\partial_r Q(r) <0$ for any $r = |(x,y)|>0$, and 
for any multi-index $\alpha$
\begin{equation}\label{prop-Q}
|\partial^\alpha Q(x,y)| \leq c(\alpha) \, \la r \ra^{-1/2} e^{- r} \quad \mbox{for any}\quad (x,y) \in \cR^2,
\end{equation}
where $\la r\ra = (1+r^2)^{1/2}$. 

Our work \cite{FHRY2} was the first result establishing the existence of blow-up solution in a higher-dimensional KdV-type family of equations, specifically in the 2D critical (cubic) Zakharov-Kuznetsov equation. Our proof was based on the method of Martel-Merle for the 1D generalized (critical) KdV equation \cite{MM-liouville, M-blowup} and, in particular, relied on certain spectral properties of a virial-type operator, which we describe below. In 2D neither the ground state $Q$ is explicit, nor the eigenfunction corresponding to the negative eigenvalue of the linearized operator $\mathcal L$ is (see Section \ref{S:Background}), therefore, in \cite{FHRY2} we investigated the spectral properties numerically, see further discussion in \S \ref{S:motivation}. 

In this paper we investigate the spectral  problem analytically and reduce it to verifying numerically only a few signs of inner products 
and certain eigenvalues, in a spirit of the 1D critical gKdV case as in \cite{MM-liouville} (where the ground state and the properties of the linearized operator are explicit).   

The main result of this paper is
\begin{theorem}[spectral estimate]
Suppose 
$1 < p \leq 3$ and $v$ satisfies the linear equation
$$\partial_t v = \mathcal{L}\partial_x v - (p-1) \frac{ \la v, (Q^p)_x \ra}{ \|Q\|_{L^2}^2} Q.$$
Suppose that $v$ satisfies (for all $t$) the orthogonality conditions
$$\la v, Q \ra = 0 \,, \qquad \la v, \nabla Q \ra =0,$$
and in the case $p=3$ \emph{in addition} satisfies
$$\la v, \Lambda Q \ra =0.$$
Then
$$- \partial_t  \int xv^2 \, dx \gtrsim_p \|v\|_{L^2}^2.
$$
\end{theorem}

\begin{remark} This theorem is proved via a lemma ($f-f^\ast$ Lemma) for a general self-adjoint operator on a Hilbert space, which is then applied to a specific setting of the 2D ZK equation (Lemma \ref{L:fstar}), showing in Corollary \ref{C:main} the items needed to be verified (signs of some inner products and eigenvectors), which are then verified numerically (in the range $1.1 \leq p \leq 3$) in Section \ref{S:results}. The limiting case as $p \to 1$ is more involved computationally and will be investigated elsewhere.   
\end{remark}

The consequence, explained briefly in \S \ref{S:2Dblowup} (for further details see \cite[\S 1.1]{FHRY2}), is 
\begin{corollary}[rigidity]\label{C:Rigit}
Let
$1< p \leq 3$.
Suppose that $w$ satisfies the linear equation
\begin{equation}\label{E:w}
\partial_t w = \partial_x \mathcal{L} w + \frac{(p^2-1) \la w, \mathcal{L} (Q^p)_x \ra}{2\int Q^{p+1}}\Lambda Q + \frac{\la w, \mathcal{L}Q_{xx} \ra}{\|Q_x\|_{L^2}^2}Q_x + \frac{\la w, \mathcal{L} Q_{xy} \ra}{\|Q_y\|_{L^2}^2} Q_y
\end{equation}
and for all $t$, satisfies the orthogonality conditions
\begin{equation}\label{E:ortho-2D}
\la w, Q^p \ra =0 \,, \qquad \la w, \nabla Q \ra =0.
\end{equation}
Suppose also that $w$ is uniformly-in-time spatially localized:
\begin{equation}\label{E:uniform}
\sup_{-\infty<t<+\infty} \| \la x \ra^{1/2+} \, w(x,y,t) \|_{L_{xy}^2} \lesssim 1.
\end{equation}
Then, in fact, $w\equiv 0$.
\end{corollary}
(The notation $\la x \ra^{1/2+}$ means that there exists $\epsilon>0$ such that this inequality holds with the power $1/2+\epsilon$.)

This is applied as follows.  In the case $p=3$, the goal is to prove blow-up for negative energy solutions near $Q(x,y)$, see the statement below in Theorem \ref{T:main}. In the case $1<p<3$, one seeks asymptotic stability of traveling waves with profile $Q(x,y)$.  In either case, argue by contradiction, and pass to weak limits in time, to construct a sequence of \emph{radiation free} solutions $\tilde u_n$ (to the nonlinear problem).  A key component of this construction is the implementation of a {\it monotonicity} estimate.  In the case $p=3$, one needs to use the integral conservation law \cite[(1.6)]{FHRY2} to show the $\tilde u_n$ solutions are global.  These solutions approach $Q$ as $n\to \infty$.  Take the radiation free solutions $\tilde u_n$, linearize around $Q$, and take a \emph{renormalized} limit of the remainder $\tilde \epsilon_n$ as $n\to \infty$.  This generates a nontrivial solution $w$ to the linear equation above. Appealing to the Corollary \ref{C:Rigit}, we obtain the contradiction.  
\smallskip

This paper is organized as follows: in Section \ref{S:Background} we define the linearized operator $\mathcal L$ and state useful identities, then we give further explanation and motivation for the spectral virial investigations in \S \ref{S:motivation}, which is followed up by a review of our 2D blow-up result (for the cubic ZK, $p=3$), and in \S \ref{S:general-p} we describe a general $p$ setting. 
In the Section \ref{S:reduction} we consider polar coordinates on $\mathbb R^2$ and explain a specific orthogonal decomposition of the whole space $L^2(\mathbb R^2)$ into the {\it even} and {\it odd} angular modes, and then the reduction of our problem to a subspace $\Omega_{1,e}$ (first even angular mode subspace). In Section \ref{S:f-fstar} we prove Lemma \ref{L:f-fstar} about spectral properties for a general self-adjoint operator on a Hilbert space and show how to apply it to our case of the Schr\"odinger-type operator $T$ (Lemma \ref{L:fstar}) and what items need to be checked (or verified numerically) in Corollary \ref{C:main}. In Section \ref{S:test} we show the benchmarks for our numerics and conclude with Section \ref{S:results} on the numerical verification of the assumptions in Corollary \ref{C:main} (for the range of $1 <p \leq 3$). 
\medskip

{\bf Acknowledgments. }
J.H. was partially supported by the NSF grant DMS-2055072.
S.R. was partially supported by the NSF grant DMS-2055130. 


\section{Background}
\label{S:Background}

\subsection{Linearized operator $\mathcal L$}

The operator $\mathcal L$, which is obtained by linearization around the ground state $Q$, is defined by
$$
\mathcal L = -\Delta +1 -p Q^{p-1}.
$$
Note that $\mathcal L(Q) = -(p-1) Q^{p}$.
For properties of $\mathcal L$, we refer the reader to Section 3 in \cite{FHR3} and the references therein.  

We record several useful identities. We start with the Pokhozhaev identities (obtained by multiplying either by $Q$ or $x \cdot \nabla Q$ the equation \eqref{E:Q} and integrating)
\begin{equation}\label{P:1}
\| \nabla Q \|_{L^2(\mathbb R^2)}^2 = \frac{p-1}2 \|Q\|^2_{L^2(\mathbb R^2)},
\end{equation}
\begin{equation}\label{P:2}
\| Q \|_{L^{p+1}(\mathbb R^2)}^{p+1} = \frac{p+1}2 \|Q\|^2_{L^2(\mathbb R^2)}.
\end{equation}
Next, defining the scaling symmetry generator 
\begin{equation}\label{E:Lambda}
\Lambda Q = \tfrac2{p-1} Q + {x} Q_x + y Q_y, \quad (x,y) \in \mathbb R^2, 
\end{equation}
we record some identities with $\Lambda$:
$$
\mathcal L(\Lambda Q) = -2 Q,
$$
\begin{equation}\label{LambdaQ}
\la Q, \Lambda Q \ra = \frac{3-p}{p-1} \| Q \|^2_{L^2(\mathbb R^2)}.
\end{equation}

\subsection{Motivation for spectral virial investigations}\label{S:motivation}

In \cite{FHRY2} we proved the existence of blow-up in finite or infinite time in the critical ($p=3$) ZK equation by introducing a new approach to study the virial operator (as well as several other novel elements such as a 2D monotonicity, weighted Gagliardo-Nirenberg inequality, usage of the $L^1$ (integral)-type conservation law under limited regularity, and also subtleties in regularization in a passage to the dual problem).  
The inspiration came from the 1D works on the critical gKdV equation in a series of papers by Martel and Merle, see \cite{MM-liouville, M-blowup, MM-KdV3}, as well as the adaption of it to the asymptotic stability result for the 2D quadratic ZK equation in \cite{CMPS}. The crucial element of this approach is the spectral property of a virial operator: it has played a fundamental role in the analysis of the 1D generalized KdV equation, see for example, \cite{MM-liouville, M-blowup, Martel-dual}, and it was also a key element in our proof \cite{FHRY} of the asymptotic stability of solitons in the 3D ZK equation, postulated by Zakharov and Kuznetsov in their original work deriving the ZK equation and discussing stability of the 3D solitons \cite{ZK}.

Since the ground state is not given explicitly in 2D (or higher dimensions), this creates extra challenges to study  spectral properties of the Schr\"odinger-type operators in 2D and higher, in particular, the ones connected to the ZK equation.  

In this paper, we show a simplified approach to verify the spectral properties, with rather simple numerically-assisted proof, which basically comes down to verifying the signs of certain inner products. Numerical verification of signs of inner products as well as funding functions that are inverses of some Schr\"odinger-type nonlinear operators can be tracked down even in the 1D KdV analysis in papers \cite{MM-liouville, M-blowup, MM-KdV3}, as well as in 2D works \cite{FHRY, CMPS}, and 3D \cite{FHRY2}. For other equations, such as the nonlinear Schr\"odinger equation or Hartree equation, verification of spectral properties for virial operators are also common, see \cite{FMR, YRZ2018, YRZ2020}. 

\subsection{Blow-up in the 2D critical (cubic) ZK} \label{S:2Dblowup}
To describe the findings in this paper, we first recall our result about the existence of blow-up in 2D critical (cubic) ZK equation
\begin{equation}\label{E:ZK}
\partial_t u + \partial_x \left(\Delta_{(x,y)} u + u^3 \right) = 0, 
\end{equation}
that we proved in \cite{FHRY2} (for later results on blow-up in 2D critical ZK equation, see, for example, \cite{Gong}, and for numerical simulations \cite{KRS}).
\begin{theorem}(\cite{FHRY2})\label{T:main}
There exists $\alpha_0>0$ such that the following holds. Suppose that $u(t)$ is an $H^1$ solution to (ZK) with $E[u]<0$ and
$0< \|u\|^2_{L^2(\mathbb R^2)} - \|Q\|^2_{L^2(\mathbb R^2)} \leq \alpha$ for some $\alpha>0$. 
Then $u(t)$ blows up in finite or infinite forward time.
\end{theorem}

Assuming that the theorem is false, we constructed a sequence of solutions $\{\tilde u_n\}$ of ZK equation \eqref{E:ZK} with negative energy $E(\tilde u_n)<0$, which is well-behaved: these solutions are uniformly-in-time $H^1$ bounded and are spatially localized, which was shown by smallness of the remainder functions $\tilde \epsilon_n$ (that were defined for some time-dependent parameters $\tilde \lambda_n$ (scaling), $\tilde x_n, \tilde y_n$ (translation) ) as
$$
\tilde \epsilon_n(x,y,t) \defeq \tilde \lambda_n(t) \, \tilde u_n \big(\tilde\lambda_n(t) x + \tilde x_n(t), \tilde \lambda_n(t) y + \tilde y_n(t), t \big)-Q(x,y)
$$
and satisfied the orthogonality conditions $\la \tilde \epsilon_n, \nabla Q\ra =0$, $\la \tilde \epsilon_n, Q^3 \ra=0$).
We then obtained a rigidity property of $\tilde u_n$'s for sufficiently large $n$, which stated that these well-behaved solutions (for large $n$) are nothing but a rescaled and shifted solitary wave $Q$. In this critical case of ZK equation $E[Q] = 0$ (and hence, $E[\tilde u_n]=0$ for large $n$), which contradicted the initial assumption of negative energy. 

For the rigidity, or nonlinear Liouville theorem, we considered a renormalized limit of the remainders $\{\tilde \epsilon_n\}$, which produced a (nontrivial) solution $w$ to a linearized (around $Q$) ZK equation, thus, reducing the problem to a linear Liouville property.
The equation that $w$ solves is given in \eqref{E:w} (with $p=3$) and it also satisfies the orthogonality conditions \eqref{E:ortho-2D} together with uniform-in-time estimate \eqref{E:uniform}. This linear Liouville property will imply that such $w\equiv 0$, and in order to show that we convert the problem to the adjoint by considering $v= \mathcal L w$. This formulation further simplifies the equation \eqref{E:w} by dropping terms with $\Lambda Q, \nabla Q$ due to orthogonality conditions, and from now on we work with $v$, which solves  
\begin{equation}\label{E:v1}
\partial_t v =  \mathcal{L}\partial_x v  -2 \alpha Q
\end{equation}
for some time dependent coefficient $\alpha (t)$, where $v$ satisfies the orthogonality conditions 
\begin{equation}\label{E:ortho-v}
\la v, Q \ra =0, \quad \la v, Q_x \ra =0, \quad \la v, Q_y\ra =0.
\end{equation}

In \cite{FHRY} we show that for any $-\infty < t< \infty$, the following inequality holds 
\begin{equation}
\label{E:v-virial-1}
\| v\|_{L_t^2H_{xy}^1} \lesssim \| \la x \ra^{1/2} v \|_{L_t^\infty L_{xy}^2},
\end{equation}
for which 
we differentiate the orthogonality condition $\la v,Q \ra =0$, yielding
$$
0= \partial_t \la v, Q \ra = \la \mathcal{L}\partial_x v, Q \ra -2 \alpha \la Q, Q \ra \quad \mbox{with} \quad 
\alpha = \frac{\la v, 3Q^2Q_x \ra}{\la Q, Q \ra},
$$
so that we can rewrite \eqref{E:v1} as
$$
\partial_t v = \mathcal{L} \partial_x v - \frac{\la v, 6Q^2Q_x \ra}{\la Q, Q\ra} Q.
$$
Now, in order to obtain \eqref{E:v-virial-1}, we compute 
\begin{equation}
\label{E:v-virial}
-\frac12 \partial_t \int_{\mathbb R^2} x v^2(x,y) \, dx\, dy = \la Bv,v\ra + \la P v,v\ra  =: \la A v,v\ra,
\end{equation}
where
\begin{equation}\label{E:B1}
B := \frac12 - \frac32\partial_x^2 - \frac12\partial_y^2  - \frac32 Q^2 - 3 x QQ_x
\end{equation}
and $P$ is the rank $2$ self-adjoint operator
\begin{equation}\label{E:P1}
Pv := \frac12 \frac{ 6\, Q^2Q_x}{\la Q, Q\ra} \la v, xQ\ra+ \frac12 \frac{xQ}{\la Q, Q\ra} \la v, 6Q^2Q_x\ra.
\end{equation}
Thus, to conclude \eqref{E:v-virial-1}, we had to verify the coercivity of the operator $A := B+P$, i.e., 
$$
\la Av,v \ra \geq c \, \la v,v \ra > 0
$$
for some positive constant $c$, 
provided that it satisfies the orthogonal conditions 
$$
\langle v,Q \rangle=\langle v, \nabla Q \rangle=0.
$$

Since the continuous spectrum of $A$ (or of $B$ \eqref{E:B1}) starts from $\frac{1}{2}$, one would need to understand the discrete spectrum below $\frac{1}{2}$. 
To do that in \cite{FHRY2}, as well as in \cite{FHRY}, we studied the discrete spectrum of $A$ numerically (see details in the appendix of \cite{FHRY2}). Another motivation for the numerically-assisted proof is that in higher dimensions the ground state $Q$ is not given explicitly, and furthermore, the rank 2 operator $P$ has a non-trivial representation. 
In \cite{FHRY2} we showed one possible numerical approach to obtain the discrete spectrum, which was based mainly on computing the eigenvalues and eigenfunctions and then checking the angles between the appropriate functions. 

Below we develop an approach, which is very minimally based on the numerics, but rather analyzes general properties of the 2D Schr\"odinger operator with potentials. 

In order to complete the proof, one needs to show that the estimate (that we obtained above)
$$ \sup_t \|\la x \ra^\alpha v \|_{L^2}^2 \gtrsim \int_t \|v\|_{H^1}^2 $$
implies the corresponding estimate for $w$.  Since $v = (1-\gamma \Delta)^{-1} \mathcal{L}w$ for $\gamma>0$ small, this amounts to proving that, for sufficiently small $\gamma>0$,
$$ \| \la x \ra^\alpha (1-\gamma \Delta)^{-1} \mathcal{L}w \|_{L^2} \sim \| \la x \ra^\alpha w \|_{L^2},
$$
$$\| (1-\gamma \Delta)^{-1} \mathcal{L} w\|_{H^1} \sim \|w\|_{H^1}.
$$
These reduce to standard commutator estimates that can be proved using Schur's test. Moreover, we have 
$$\| w\|_{H^1}^2 \gtrsim \la \mathcal{L}w,w \ra.
$$
These estimates do not rely on the orthogonality conditions and together give
$$ \sup_t \|\la x \ra^\alpha w \|_{L^2}^2 \gtrsim \int_t \; \la \mathcal{L}w,w \ra \; dt.
$$
Note that $\la \partial_t \mathcal{L}w,w \ra =0$. Via the orthogonality conditions for $w$, $\la \mathcal{L}w,w \ra >0$.  This implies $w=0$.

We point out that even in the 1D case of the critical gKdV equation, when $Q$ is given explicitly as a $\sech$ function, and where the spectral properties of certain classical Schr\"odinger-type operator are available and well-known, certain function (an inverse of the ground state $Q$ under a specific second order nonlinear operator) had to be computed numerically as well as signs of some inner products also had to be checked numerically, see (177), Remark on p. 415 and the proof of Lemma 26 in \cite{MM-liouville}. 
\smallskip

We next consider a general case of the 2D Zakharov-Kuznetsov equation with any nonlinearity $p>1$ and set up a general virial operator, for which we then investigate the spectral properties.

\subsection{General nonlinearity $p$}\label{S:general-p}

Recall the direct problem from \eqref{E:w}
$$
\partial_t w = \partial_x \mathcal{L} w + \alpha \Lambda Q + \beta Q_x + \gamma Q_y
$$
for time-dependent coefficients $\alpha$, $\beta$, and $\gamma$, with  $w$ satisfying the orthogonality conditions 
$$ 
\la w, Q^p\ra =0, \quad \la w, Q_x \ra =0, \quad \la w, Q_y \ra =0.
$$

Reformulating it as an adjoint problem, $v=\mathcal L w$, we have 
\begin{equation}\label{E:v}
\partial_t v =  \mathcal{L} \partial_x v -2 \alpha Q 
\end{equation}
with the time-dependent coefficient $\alpha$ given by 
$$
\alpha = \frac{p-1}2 \, \frac{\la v, (Q^p)_x \ra}{\la Q,Q \ra},
$$
with $v$ satisfying the orthogonality conditions 
\begin{equation}\label{E:ortho-v-p}
\la v, Q \ra =0, \quad \la v, Q_x \ra =0, \quad \la v, Q_y \ra =0.
\end{equation}

Differentiating in time the quantity $\int x v^2$ from \eqref{E:v-virial} and using the equation \eqref{E:v}, we obtain
\begin{equation}\label{E:virial}
-\frac12 \partial_t \int_{\mathbb R^2} x v^2 = \la Bv,v \ra + \la P v, v \ra,
\end{equation}
where 
\begin{align}
B = -\frac32 \partial_{xx} - \frac12 \partial_{yy} + \frac12 - \frac{p}2 \, Q^{p-1} - \frac{p}2 \, x (Q^{p-1})_x, \label{E:B}\\
P v = \frac{p-1}2 \frac{(Q^{p})_x}{\|Q\|^2_2} \la v ,\,xQ  \ra +  \frac{p-1}2 \, \frac{x Q}{\|Q\|^2_{L^2}} \la v, (Q^p)_x \ra. \label{E:P}
\end{align}

Recalling the linearized operator $\mathcal L$, we write $2B$ as 
\begin{equation}\label{E:decompB+P}
2B = \mathcal L + \mathcal E, 
\end{equation}
noting that $\mathcal{E} \defeq -2 \partial_x^2 - x \,p(Q^{p-1})_x$ is a positive definite operator (the second term is a non-negative function).  The positivity of $\mathcal{L}$, under the relevant orthogonality conditions, is classical and summarized in the following lemma.  Note that the proof of Proposition 13.1 in \cite{FHRY2} implies that in the case $p=3$ we can also assume $\la v, \Lambda Q \ra =0$.

\begin{lemma}\label{L:ortho}
Suppose $v$ satisfies the orthogonality conditions \eqref{E:ortho-v-p} for $1<p<3$, and for $p=3$ in addition satisfies 
\begin{equation}\label{E:ortho-p3}
\la v, \Lambda Q \ra =0, 
\end{equation}
where $\Lambda Q$ is defined in \eqref{E:Lambda}. 
Then there exists $c_p>0$ such that 
$$
\la \mathcal{L} v, v \ra \geq c_p \|v \|_{L^2}^2.
$$
\end{lemma}

\begin{proof}
In \cite{FHR3} (see Theorem 3.1 and Lemmas 3.3-3.6) we summarized several known positivity estimates for the operator $\mathcal{L}$. In Weinstein \cite[Prop 2.7, Prop. 2.8(b)]{W85}, it is shown that 
$$\inf_{\substack{v \\ \la v, Q \ra =0}} \la \mathcal{L} v,v \ra =0  \quad \text{and} \quad \ker \mathcal{L} = \spn \{\nabla Q \} \,.$$

First, consider the case $p=3$ in which $\la \Lambda Q, Q \ra =0$.  Let 
$$
\alpha = \inf \{\, \la 
\mathcal{L} v, v \ra \; : \;  v \text{ satisfies }\la v, \Lambda Q \ra =0\,, \; \la v, Q \ra =0\,,\; \la v, \nabla Q \ra=0\,,\; \|v\|_{L^2}=1 \, \}
$$
Clearly $\alpha \geq 0$.   Arguing by contradiction, suppose $\alpha=0$.   Let $v$ be a minimizer.  Then $v$ satisfies the Euler-Lagrange equation
$$
\mathcal{L} v = \beta Q + \gamma_1 Q_x + \gamma_2 Q_y + \eta \Lambda Q + \omega v.
$$
Pairing with $v$, we find that $\omega=\alpha=0$. Pairing with $Q_x$ and $Q_y$, we find that $\gamma_1=0$ and $\gamma_2=0$.  Since $\mathcal{L}\Lambda Q = -2Q$ and $\la Q, \Lambda Q \ra =0$, pairing with $\Lambda Q$ yields $\eta =0$.  Thus,
$$
\mathcal{L}v = \beta Q. 
$$
Since $\mathcal{L}\Lambda Q = -2Q$, we have
$$
v = -\frac12 \beta \Lambda Q, 
$$
which implies $\beta =0$, a contradiction, and hence, $\alpha >0$.

Now consider $1<p<3$.  Then $\la \Lambda Q, Q \ra >0$. We repeat the above proof with $\eta =0$ to get
$$
\inf \{ \, \la 
\mathcal{L} v, v \ra \; : \; v \text{ satisfies } \la v, Q \ra =0\,,\; \la v, \nabla Q \ra=0\,,\; \|v\|_{L^2}=1 \, \} >0,
$$
finishing the proof. 
\end{proof}

\begin{remark}
This lemma implies that $\mathcal{L}$ is coercive on all subspaces $\Omega_{k,e}$ and $\Omega_{k,o}$, defined in the next section.  
\end{remark}

\section{Reduction to $\Omega_{1,e}$}\label{S:reduction}

On $L^2(\mathbb{R}^2)$, define the operators using polar coordinates as
$$[\Pi_{k,e} w] (r,\theta) =\frac{1}{\pi} \left(\int_{\varphi=0}^{2\pi} w(r,\varphi) \cos (k\varphi) \, d\varphi\right) \cos (k\theta),$$
$$[\Pi_{k,o} w] (r,\theta) =\frac{1}{\pi} \left(\int_{\varphi=0}^{2\pi} w(r,\varphi) \sin (k\varphi) \, d\varphi\right) \sin (k\theta).$$

\newcommand{\ran}{\operatorname{ran}}

These operators are orthogonal projections on $L^2(\mathbb{R}^2)$ (each  satisfies $\Pi^2=\Pi$ and $\Pi^*=\Pi$).  Let
$$\Omega_{k,e} = \ran \Pi_{k,e} \,, \qquad \Omega_{k,o} = \ran \Pi_{k,o}.$$
(Note that $\Pi_{0,o}=0$ and $\Omega_{0,o}=\{0\}$.)  
Then we have the even and odd angular mode orthogonal decomposition
$$
L^2(\mathbb{R}^2) = \bigoplus_{k=0}^\infty (\Omega_{k,e}\oplus \Omega_{k,o}).
$$
For convenience, we use the following notation 
$$
\lla u,v \rra = \int_0^\infty u(r)\bar v(r) \, r \, dr.
$$ 
Then the orthogonality conditions \eqref{E:ortho-v-p} become
\begin{align}
& 
\llangle v , Q \rrangle = 0 \quad \mbox{for} ~ v \in \Omega_{0,e} \quad \mbox{and} \label{E:ortho-v-1}\\
&
\llangle v ,\partial_r Q \rrangle = 0 \quad \mbox{for}~~ v(r) \cos\theta \in \Omega_{1,e}, \label{E:ortho-v-2} \\
&
\llangle v ,\partial_r Q \rrangle = 0 \quad \mbox{for}~~ v(r) \sin\theta \in \Omega_{1,o}.  \notag
\end{align}

In light of Lemma \ref{L:ortho} when $p=3$ we also require (here, $\Lambda Q =  Q + r \partial_r Q$) 
\begin{equation}\label{E:p3}
\llangle v, \Lambda Q \rrangle =0 \quad \mbox{for} \quad v\in \Omega_{0,e}.
\end{equation}

The Laplacian is invariant on each of these subspaces, and the finite rank operator $P$ acts only on $\Omega_{1,e}$.  Specifically,
$$\forall \; k\geq 0 \,, \quad \mathcal{L}\Pi_{k,e} = \Pi_{k,e}\mathcal{L} \quad \text{and} \quad \mathcal{L}\Pi_{k,o} = \Pi_{k,o}\mathcal{L}$$
$$\forall \; k\geq 0 \,, \; P\Pi_{k,o} = 0 \,, \qquad \forall \; k=0,2,3,\ldots \,, \; P\Pi_{k,e} = 0 \,, \qquad P= P\Pi_{1,e} = \Pi_{1,e} P $$
Thus,
$$\mathcal{L}+\mathcal{E}+2P = \sum_{k\geq 0 \,, k\neq 1} \Pi_{k,e}\mathcal{L}\Pi_{k,e} + \sum_{k\geq 0} \Pi_{k,o}\mathcal{L}\Pi_{k,o} +\mathcal{E} + (\Pi_{1,e}\mathcal{L}\Pi_{1,e} + 2\Pi_{1,e}P \Pi_{1,e}). $$
Since $\mathcal{L}$, subject to the orthogonality conditions \eqref{E:ortho-v-p} (and \eqref{E:ortho-p3} when $p=3$), is positive, 
it follows that the first two terms are positive (subject to the orthogonality conditions \eqref{E:ortho-v-1}-\eqref{E:ortho-v-2} and \eqref{E:p3} when $p=3$).  As mentioned above, $\mathcal{E}$ is positive.   It remains to show that 
$$\Pi_{1,e}\mathcal{L}\Pi_{1,e} + 2\Pi_{1,e}P \Pi_{1,e}$$
is positive (subject to the orthogonality conditions). Consider the densely defined\footnote{The operator $U$ is actually defined on the subspace $H^2((0,\infty),rdr) \cap H^1_0((0,\infty),rdr)$ which is dense in $L^2((0,\infty),rdr)$, whereas $V$ is everywhere defined on $L^2((0,\infty),rdr)$.} operators
\begin{equation}\label{Def:U-V}
U,V : L^2( (0,+\infty), rdr) \to L^2( (0,+\infty), rdr)
\end{equation}
by
\begin{equation}\label{E:U-V-def}
(U v )(r) = \frac{\Pi_{1,e} \mathcal{L} \Pi_{1,e} (v(r)\cos\theta)}{\cos \theta} \,, \qquad (V v )(r) = \frac{ 2 \Pi_{1,e} P \Pi_{1,e} (v(r)\cos\theta)}{\cos \theta}.
\end{equation}
It suffices to show that $T:=U+V$ is positive subject to the converted orthogonality condition \eqref{E:ortho-v-2}.
From \eqref{E:U-V-def}, we have
$$U  = -\partial_r^2 - \frac1r\partial_r +\frac{1}{r^2} + 1 - pQ^{p-1}$$
and, recalling the expression for $2P$ from \eqref{E:P}, we also have
\begin{equation}\label{E:Vv}
(Vv)(r) = (p-1) \frac{\lla v, r Q\rra }
{\lla Q, Q \rra} \, (Q^p)_r
+ (p-1) 
\frac{\lla v,  (Q^p)_r \rrangle}{\lla Q,Q \rra}
\, rQ,
\end{equation}
where 
$${\lla  v, r Q\rra } = \int_0^\infty v(r) Q(r) r^2 \,dr \quad 
\mbox{and} \quad 
\lla v,  (Q^p)_r \rra=\int_0^\infty v(r) \partial_r([Q(r)]^p) \,r \, dr.
$$ 
Renormalizing $(Q^p)_r$ and $rQ$ as 
$h_1 = c (Q^p)_r$ and $h_2 = c\, rQ$ with $c = \sqrt{\frac{p-1}{\lla Q,Q\rra} \,} \equiv \sqrt{\frac{2\pi(p-1)}{\|Q\|^2_{L^2({\mathbb R^2})}}}$, we rewrite \eqref{E:Vv} as 
\begin{equation}\label{E:Vv-normalized}
(Vv)(r) = \lla v, h_2\rra  \, h_1 + \lla v, h_1\rra  \, h_2.
\end{equation}

It is convenient for the numerical simulations to reduce the singularity of $U$ near $r=0$ by converting between $U$ and an operator $\hat{U}:  L^2( (0,+\infty), r^3dr) \to L^2( (0,+\infty), r^3dr) $ 
given by
$$\hat{U} = - \partial_r^2 - \frac{3}{r}\partial_r + 1 - pQ^{p-1}.
$$
Direct computation confirms that
$$U = r\, \hat{U}\, r^{-1}.
$$

Similarly, converting $V$ to $\hat{V}: L^2( (0,+\infty), r^3dr) \to L^2( (0,+\infty), r^3dr) $ via $V=r\hat V r^{-1}$ yields
$$
(\hat{V} v)(r) = p(p-1) \frac{\lla  v, r^2 Q\rra }
{\lla Q, Q \rra} \frac{Q^{p-1} Q_r}{r}
+ (p-1) 
\frac{\lla v,  r (Q^p)_r \rrangle}{\lla Q,Q \rra}
\, Q.
$$
In the numerical implementation, it is more convenient to use $T = \hat{U} + \hat{V}$, however, for simplicity of exposition in the paper, we formulate the results for $T = U+V$. 

\section{$f$-$f^*$ Lemma}\label{S:f-fstar}

\newcommand{\ind}{\operatorname{ind}}
\newcommand{\proj}{\operatorname{proj}}

Borrowing language typically applied to quadratic forms, we define the \emph{index} (of negativity) of a Schr\"odinger operator to be the cumulative number of dimensions of the negative eigenspaces, or, the number of negative eigenvalues counted with multiplicities.

The next lemma deals with a general self-adjoint operator on a Hilbert space and some of its spectral properties. Later on it will be applied to $T=U+V$ and checking $\ind T\leq 1$ and the bottom of the positive spectrum is strictly positive.

\begin{lemma}[$f$-$f^*$ Lemma]
\label{L:f-fstar}
Let $H$ be a Hilbert space with an inner product $\la \cdot, \cdot \ra$, $T:H\to H$ is a densely defined self-adjoint operator, and $f\in H$ is given.  
\begin{enumerate}[label=(\alph*),left=0pt]
\item Assume the following:
\begin{itemize}
\item  $\ker T=\{0\}$.
\item $\ind T \leq 1$.
\item There exists $f^*\in H$ such that $Tf^*=f$ satisfying $\la f, f^*\ra <0$.  
\end{itemize}
Then $T$ is positive on $(\spn \{f\})^\perp$.  This means that if $v\neq 0$ satisfies $\la v, f \ra =0$, then $\la Tv,v \ra >0$.    

\item Let $\mu \geq 0$ be the bottom of the nonnegative spectrum of $T$.  If, in addition to the assumptions of (a), we further assume that $\mu>0$, then 
$$\mu^* = \inf \{ \; \la Tv,v \ra \; : \; \|v\|=1 \text{ and } \la v,f \ra =0 \; \}$$
satisfies $\mu^*>0$.
\end{enumerate}
\end{lemma}

\begin{remark}
Note that we do need the additional assumption in (b) to obtain the conclusion indicated there.  For example, take $T=-\partial_x^2 - 2 \sech^2 x$, a standard Poeschl-Teller Schr\"odinger operator, with $f=\sech x$.  Then $T$ meets the hypotheses of (a) with $f^*=-\sech x$ but nevertheless $\mu=\mu^*=0$.   Thus, an extra assumption is needed to upgrade from positivity to uniform positivity.
\end{remark}

\begin{proof}
(a) Since $\ind(T) \leq 1$, either $\ind(T)=0$ or $\ind(T)=1$. Since $\la Tf^*,f^*\ra = \la f,f^*\ra <0$, it follows that $\ind(T)\neq 0$, so we must have $\ind(T)=1$.  Thus, there exists $e\in H$, $\|e\|=1$ such that $Te=-\lambda e$, where $\lambda>0$.  Note that $\la f^*, e \ra \neq 0$. Indeed, if $\la f^*, e \ra =0$, then $\la Tf^*,f^* \ra \geq 0$, but $\la Tf^*,f^*\ra = \la f, f^*\ra <0$ by assumption.

Given $v\in H$ such that $\la v,f \ra =0$ and $v\neq 0$, we aim to show that  $\la Tv,v \ra >0$. Let 
$$
w = v - \frac{\la v, e \ra}{\la f^*,e \ra } f^* \in (\spn\{e\})^\perp,
$$
which is well-defined since $\la f^*,e\ra \neq 0$.
Since $Tf^*=f$ and $\la v, f \ra =0$, it follows that
$$
\la Tv,v\ra =  -\frac{|\la v,e\ra|^2\la f,f^* \ra}{|\la e, f^*\ra|^2} + \la Tw,w\ra. 
$$
Since $w\in (\spn\{e\})^\perp$ and $\ker T = \{0\}$, it follows that $\la Tw,w\ra >0$ if $w\neq 0$.  If $\la v, f \ra =0$ and $v\neq 0$, then either 
\begin{enumerate}[label=($\bullet$), left=0pt]
\item $\la v, e \ra \neq 0$, in which case we know the first term is strictly positive and the second term is nonnegative, so $\la Tv,v\ra >0$.
\item $\la v, e \ra =0$, in which case $w=v$, and thus, $w\neq 0$ and $\la Tv,v\ra = \la Tw,w\ra >0$.
\end{enumerate}
This completes the proof of (a).

(b)  As shown above, there is exactly one negative eigenvalue with eigenspace spanned by some $e\in H$.  Also, we have assumed that $\ker T =\{0\}$. Hence $\mu$, defined as the bottom of the nonnegative spectrum of $T$, is, in fact,
$$\mu = \inf \{ \; \la Tv,v \ra \; : \; \|v\|=1 \text{ and } \la v,e \ra =0 \; \}.
$$
Let $\{v_n\}$, with $\|v_n\|=1$, $\la v_n,f\ra =0$, be a minimizing sequence for $\mu^*$, i.e.,
$$\la Tv_n,v_n \ra \to \mu^* \text{ as } n\to \infty$$
Suppose there exists a subsequence $\{ v_{n_k}\}$ of $\{v_n\}$ such that $\la v_{n_k},e \ra \to 0$ as $k\to \infty$.  Let
$$w_k = v_{n_k} - \la v_{n_k},e \ra e \in (\spn\{e\})^\perp.
$$
Then 
$$\la Tw_k,w_k \ra =  \la Tv_{n_k},v_{n_k} \ra + \lambda |\la v_{n_k},e\ra|^2 \to \mu^* \text{as }k\to \infty.$$
However, since $w_k \in (\spn \{e\})^\perp$, for all $k$ we have 
$$\la Tw_k,w_k \ra \geq \mu \|w_k\|^2 \to \mu.$$
Thus, $\mu^* \geq \mu>0$.

Now suppose that there does not exist a subsequence $\{v_{n_k}\}$ of $\{v_n\}$ such that $\la v_{n_k},e \ra \to 0$ as $k\to \infty$.  Since $|\la v_n,e\ra| \leq 1$, there must exist a subsequence $v_{n_k}$ such that $\la v_{n_k},e\ra \to \alpha$ for some $\alpha$, it is just that we now must have $\alpha \neq 0$.  In this case, we return to the end of the proof of (a), and define 
$$w_k = v_{n_k} - \frac{\la v_{n_k}, e \ra}{\la f^*,e \ra } f^* \in (\spn\{e\})^\perp$$
and
$$\la Tv_{n_k},v_{n_k}\ra =  -\frac{|\la v_{n_k},e\ra|^2\la f,f^* \ra}{|\la e, f^*\ra|^2} + \la Tw_k,w_k\ra \geq -\frac{|\la v_{n_k},e\ra|^2\la f,f^* \ra}{|\la e, f^*\ra|^2}.$$
Sending $k\to \infty$, we obtain
$$\mu^* \geq -\frac{|\alpha|^2\la f,f^* \ra}{|\la e, f^*\ra|^2}>0.$$
So in either case, we obtain $\mu^*>0$.
\end{proof}

\begin{lemma}[Construction of $f^*$]\label{L:fstar}
Let $U,V:  L^2( (0,+\infty), rdr) \to L^2( (0,+\infty), rdr) $ be the following self-adjoint Schr\"odinger operator 
$$U = -\partial_r^2 - \frac{1}{r}\partial_r + \frac{1}{r^2} +1 - pQ^{p-1}$$
(on a domain with boundary condition $0$ at $r=0$) and self-adjoint finite rank operator
$$
V v  = \lla v,h_1 \rra h_2 + \lla v,h_2 \rra h_1.
$$

Let $T=U+V$ and let $F  = \ker U= \spn \{f\}$. Suppose that $\{f,h_1,h_2\}$ is a linearly independent set.  Denote
\begin{equation}\label{E:h-tilde}
\tilde h_j = h_j - \frac{\lla h_j,f\rra f}{\lla f,f \rra } = \proj_{F^\perp} h_j \,, \quad j=1,2.
\end{equation}
Let $h_j^*$ be the unique function in $F^\perp$ such that $Uh_j^*=\tilde h_j$, $j=1,2$. 

Let
\begin{equation}\label{E:M}
M = \begin{bmatrix}
\lla Vf,f \rra & \lla \tilde h_1 + V h_1^* ,f \rra & \lla \tilde h_2 + V h_2^*,f \rra\\
\lla Vf, \tilde h_1 \rra & \lla \tilde h_1 + V h_1^* ,\tilde h_1 \rra & \lla \tilde h_2 + V h_2^*, \tilde h_1 \rra\\
\lla Vf, \tilde h_2 \rra & \lla \tilde h_1 + V h_1^* , \tilde h_2 \rra & \lla \tilde h_2 + V h_2^*, \tilde h_2 \rra
\end{bmatrix}.
\end{equation}
Then
\begin{enumerate}[left=0pt,label=(\alph*)]
\item $\ker T = \{0\}$ if and only if $M$ is non-singular.
\item
The equation  $Tf^*=f$ holds if and only if
\begin{equation}
\label{E:fstar}
f^* = \beta f + \gamma_1 h_1^* + \gamma_2 h_2^*,
\end{equation}
where $(\beta,\gamma_1,\gamma_2)$ satisfy
\end{enumerate}

\begin{equation}
\label{E:systemforfstar}
\begin{bmatrix}
\lla f,f \rra\\
0 \\
0 
\end{bmatrix}
=
M
\begin{bmatrix}
\beta\\
\gamma_1\\
\gamma_2
\end{bmatrix}.
\end{equation}

\end{lemma}

\begin{remark}
We emphasize that from Lemma \ref{L:fstar} and on we only use the one dimensional $\lla \cdot,\cdot \rra$ as the inner product on the $L^2((0, +\infty), r dr)$ space. 
\end{remark}

\begin{proof}
Since $U$ is self-adjoint, the range of $U$ is exactly $F^\perp$, and hence, $\proj_{F^\perp} U = U$. Note also that $\proj_{F^\perp} f = 0$.    

To prove the part (b), we suppose that $f^*$ is given and the equation $(U+V)f^* = f$ holds.  Take $\proj_{F^\perp}$ of both sides to obtain
$$
Uf^*= -\proj_{F^\perp} Vf^* \in \spn \{ \tilde h_1, \tilde h_2\}.
$$
Thus, there exists $\beta, \gamma_1, \gamma_2 \in \mathbb{C}$ such that 
$$
f^* = \beta f + \gamma_1 h_1^* + \gamma_2 h_2^*, 
$$
which implies both
$$
Uf^* = \gamma_1 \tilde h_1 + \gamma_2 \tilde h_2
$$
and
$$
Vf^* = \beta Vf + \gamma_1 Vh_1^* + \gamma_2 Vh_2^*.
$$
Substituting, the equation $(U+V)f^* = f$ becomes
$$
\beta Vf + \gamma_1(\tilde h_1 + Vh_1^*) + \gamma_2(\tilde h_2 + Vh_2^*) = f.
$$
Every term in this equation belongs to $\spn\{f,\tilde h_1, \tilde h_2\}$, and thus the equation is equivalent to the collection of three equations relating the coefficients with respect to this basis. These equations can be derived by taking the inner product with each of $f, \tilde h_1$, $\tilde h_2$, which yields \eqref{E:systemforfstar}.

To show part (a), we proceed in the same manner as above, replacing $f$ with $0$ in \eqref{E:systemforfstar}, which shows that $M$ must be non-singular.  
\end{proof}

\begin{corollary}\label{C:main}
In order to apply Lemma \ref{L:f-fstar} to $T=U+V$, we need only check numerically that (in the notation of Lemma \ref{L:fstar})
\begin{enumerate}[left=0pt,label=(\alph*)]
\item $\ind V=1$,
\item $M$ is invertible, and
\item when $\beta$, $\gamma_1$, $\gamma_2$ are obtained by solving \eqref{E:systemforfstar}, then $f^*=\beta f + \gamma_1 h_1^* + \gamma_2 h_2^*$ satisfies $\lla f^*,f \rra <0$.
\end{enumerate}
\end{corollary}
\begin{proof}
We know that $\ker U = \spn \{f\}$, where $f=Q_r$.  Since $Q_r$ has no zeros in $(0,+\infty)$, it is the ground state of $U$, by the Sturm oscillation theorem.  In other words, $\sigma(U) \subset [0,+\infty)$.    Thus, $\ind U=0$ and $\ind T \leq \ind U + \ind V\leq 1$.

Moreover, by the Weyl criterion, the essential spectrum of $U$ is $[1,+\infty)$.  Since the potential of $U$ is rapidly decaying, there are only a finite number of eigenvalues away from the essential spectrum, and these (if there are any) must lie in $(0,1)$.  Thus, there is a gap between $0$ and the next positive eigenvalue.

Thus, all hypotheses of Lemma \ref{L:f-fstar} are confirmed.
\end{proof}

\section{Benchmarks testing}\label{S:test}

To verify hypotheses of Corollary \ref{C:main} we first find the ground state $Q$ either by a shooting method or Newton's iterations, which is standard and gives sufficient accuracy (we emphasize that a high accuracy is not needed in this work, as we only need to verify the signs of several quantities, and in fact, we want to show that even with minimal computational efforts, one can do this numerically-assisted proof). 
\begin{figure}[!ht]
\includegraphics[width=0.8\hsize,height=0.5\hsize]{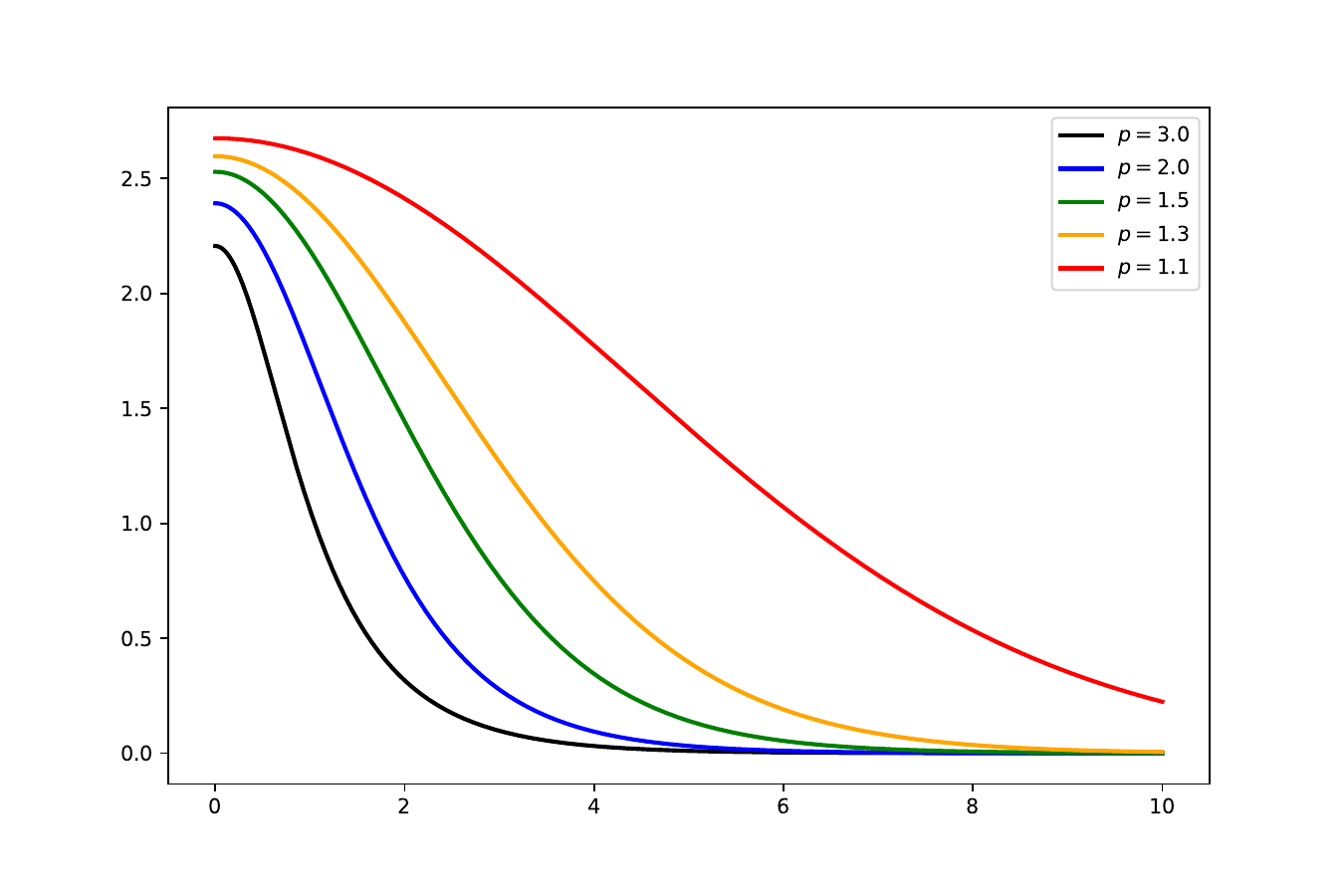}
\caption{2D ground state $Q$ for \eqref{E:Q} for different values of nonlinearity $p$.} 
\label{F:Q3}
\end{figure}
 
We show plots of the 2D (radial) ground state solutions $Q$ of \eqref{E:Q} in Figure \ref{F:Q3} for $p=2$ (left) and $p=3$ (right). 
For these computations, we choose a large enough domain, so that the boundary values of $Q$ would be on the order of at least $10^{-9}$, the 4th order (via shooting) or 2nd order (via Newton's) coupled with the Taylor expansion to the 4th order near zero, with a step size on the order of $0.001$.   
\bigskip

We also perform check of Pokhozhaev identities in \eqref{P:1}-\eqref{P:2} and \eqref{LambdaQ} (in the whole 2D space with the usual inner product $\la \cdot, \cdot \ra$), and confirm that the matching holds to a sufficient order, see Table \ref{T:1}.
{\small
\begin{table}[h!]
\centering
\begin{tabular}{|c|c|c|c|c|c|}
\hline
\textbf{Quantities} & $p = 3$ & $p = 2$ & $p=1.5$ & $p=1.3$ & $p=1.1$ \\
\hline
\hline
$\|\nabla Q\|^2_{L^2}$ & 11.7009 & 15.5016 &  18.5759 & 20.1870 & 22.1081\\
\hline
$\| Q \|^2_{L^2}$ & 11.7009 &  31.0032 &  74.3037 & 134.5797 & 442.1614 \\
\hline
$\|\nabla Q\|^2_{L^2} - \frac{p-1}2 \| Q\|^2_{L^2}$ & $\sim 10^{-9}$ & $\sim 10^{-8}$ & $\sim 10^{-9}$ & $\sim 10^{-9}$ & $\sim 10^{-9}$
\\
\hline
$\| Q\|^{p+1}_{L^{p+1}}$ & 23.4018 & 46.5048 & 92.8797 & 154.7666 & 464.2694\\
\hline
$\| Q\|^{p+1}_{L^{p+1}} - \frac{p+1}2 \| Q\|^2_{L^2}$ & $\sim 10^{-10}$  & $\sim 10^{-8}$ & $\sim 10^{-11}$ & $\sim 10^{-9}$ & $\sim 10^{-9}$
\\
\hline
$\la \Lambda Q, Q \ra$  & $\sim 10^{-12}$ &   
31.0032 & 222.91 & 762.62 &  8401.07
\\
\hline
$\la \Lambda Q, Q \ra -  \frac{3-p}{p-1} \, \| Q\|^2_{L^2}$ & $\sim 10^{-12}$ & $\sim 10^{-12}$ & $\sim 10^{-13}$ & $\sim 10^{-12}$ & $\sim 10^{-12}$
\\
\hline
\end{tabular}\smallskip
\caption{Pohozhaev quantities \eqref{P:1}-\eqref{P:2}-\eqref{LambdaQ} 
for different values of $p$.}
\label{tab:quantities}
\end{table}
}

Next we compute operators $U$ and $V$, recalling that now we only deal with the one dimensional inner product $\lla \cdot,\cdot \rra$.  
For the operator $U$ (or $\hat U$) we use a finite difference method, and for the operator $V$ (or $\hat V$) we use Simpson's rule.
To test our code, we recall $h_1$ and $h_2$ functions from \eqref{E:Vv-normalized} (and the sentence before that) and define their images under the operator $U$, to check that the numerical values for the functions $U(h_j)$, $j=1,2$, coincide with the exact values given by the formulas (recall $c = \sqrt{\frac{p-1}{\lla Q,Q \rra } \, }$  ):
\begin{align}
& \qquad h_1  = c \,(Q^p)_r \equiv c \, p \, Q^{p-1} Q_r, \notag \\ 
& U(h_1) = -p(p-1)\, c\,  \Big[ (p-2)Q^{p-3}Q^3_r +3 Q^{p-1}Q_r- 3 Q^{2(p-1)} Q_r - \frac{2}{r} Q^{p-2}Q_r^2 \Big] = :w_1, \label{Uh1}\\
& \qquad h_2  = c \, r Q, \notag \\
& U(h_2) = -2 c\, Q_r - (p-1) c\, r Q^p = :w_2. \label{Uh2}
\end{align}

In Figure \ref{F:directU} we show the exact and computed values of functions $U(h_1)$ and $U(h_2)$ for $p=3$. 
\begin{figure}[ht!]
\noindent\includegraphics[width=1.0\hsize,height=0.4\hsize]
{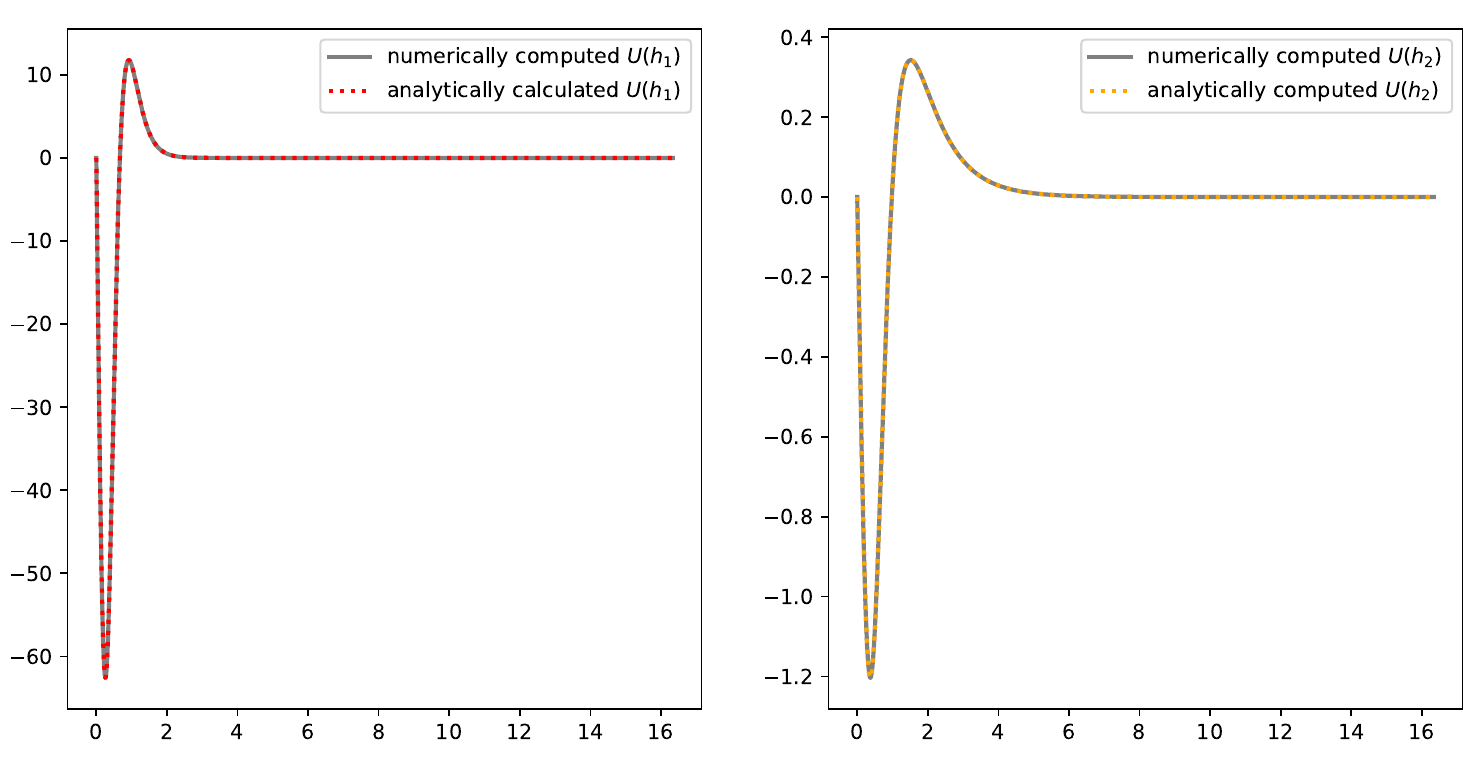}
\caption{Exact and computed functions $U(h_1)$ (left) and $U(h_2)$ (right), $p=3$.} \label{F:directU}
\end{figure}

We then test the inverse operator $U^{-1}$, by considering the functions $w_1$ and $w_2$ from \eqref{Uh1} and \eqref{Uh2}, and comparing $U^{-1}(w_1)$ with the projection onto $F^\perp$ of $h_1$, which is denoted by $\tilde h_1$ in \eqref{E:h-tilde} in the left of Figure \ref{F:inverseU} and $U^{-1}(w_2)$ with the projection onto $F^\perp$ of $h_2$, called $\tilde h_2$ in \eqref{E:h-tilde}, on the right of the same figure.  In both Figures \ref{F:directU} and \ref{F:inverseU} one can see the 
matching of the exact and numerical values. 

\begin{figure}[ht!]
\includegraphics[width=1.1\hsize,height=0.42\hsize]{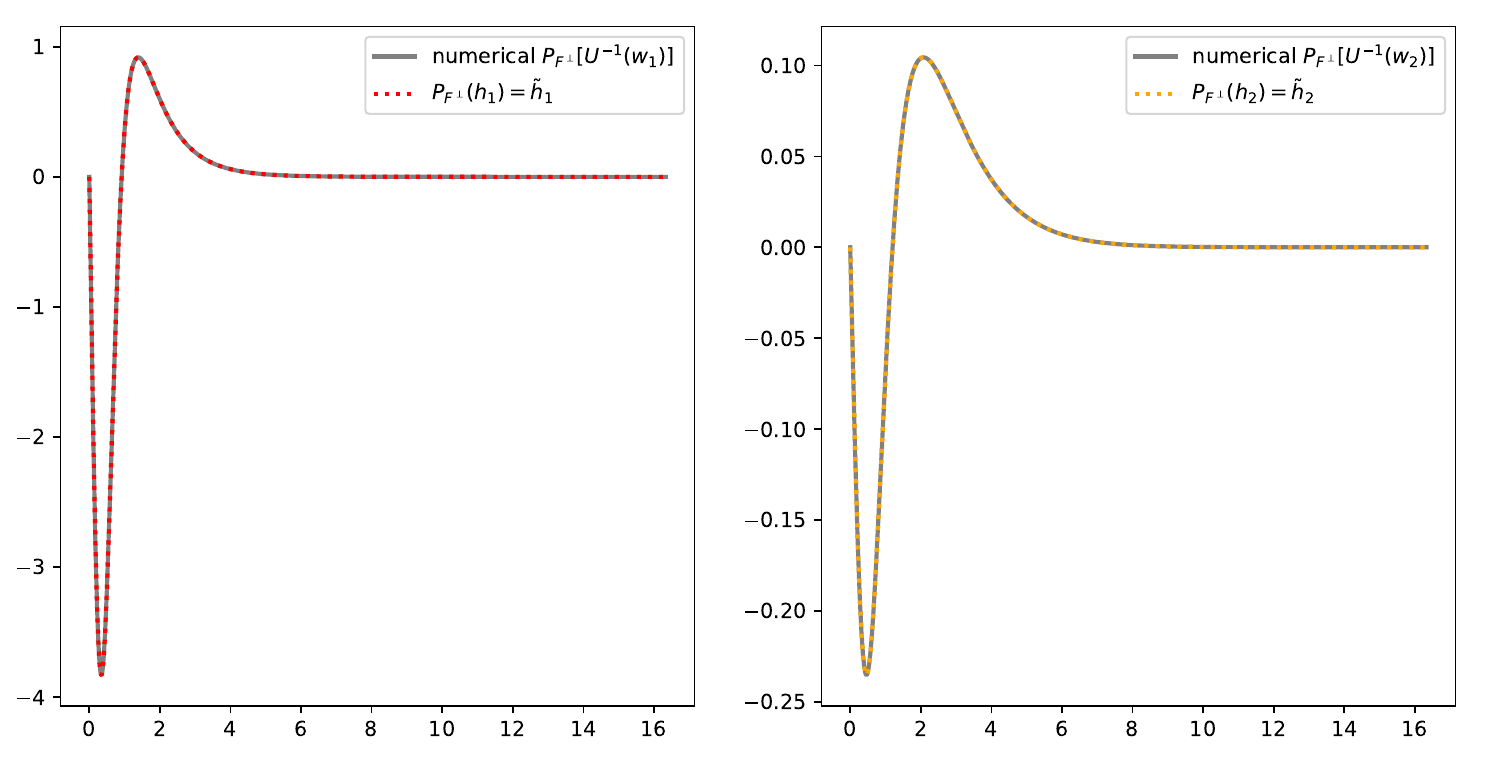}
\caption{Exact and computed 
functions $U^{-1}(w_1)$ (left) and $U^{-1}(w_2)$ (right), $p=3$.} \label{F:inverseU}
\end{figure}

We test other values of $p$ between 1 and 3, to be precise,  $p \in [1.1,3]$, we obtain similar results.

\section{Results: verification of hypotheses of Cor. \ref{C:main}}\label{S:results}

We first discuss the 2d cubic (critical) ZK equation and verification of items (a)-(c) of Corollary \ref{C:main}. 
Afterward, in Section \S \ref{S:R2}, we provide in Table \ref{T:1} the list of values for the other values of $p$. 

\subsection{2d critical ZK ($p=3$)}
For part (a) we check the index of $V$, for that, we compute the matrix 
\begin{equation}\label{E:B-matrix}
B = 
\begin{bmatrix}
\lla Vh_1,h_1 \rra & \lla V h_2,h_1 \rra \\
\lla V h_1, h_2 \rra & \lla V h_2, h_2 \rra 
\end{bmatrix},
\end{equation}
which in this case of $p=3$ is given by
$$
B
= \begin{bmatrix}
-16.8353 & 4.2439 \\
 4.2439 & -0.7219
\end{bmatrix},
$$
and it has the following eigenvalues
$$ 
(\lambda_1, \lambda_2) = (-17.885, 0.327).
$$
As it has one negative and one positive eigenvalue, we easily conclude that $\ind V = 1$, and hence, the part (a) is verified. 

For part (b), we compute the matrix $M$ in \eqref{E:M}
and its determinant, 
\begin{equation}\label{M:p3}
M = \begin{bmatrix}
-4.831 & -0.218 & -0.006 \\
-1.644 & 4.314 & 0.504 \\
-0.076 & 0.441 & 0.079
\end{bmatrix},
\qquad 
\mbox{det}\, M =  -0.59 < 0,
\end{equation}
and therefore, the matrix $M$ is non-singular, thus verifying the item (b) of Corollary \ref{C:main}. 

\begin{figure}[ht!]
\includegraphics[width=0.5\hsize,height=0.42\hsize]{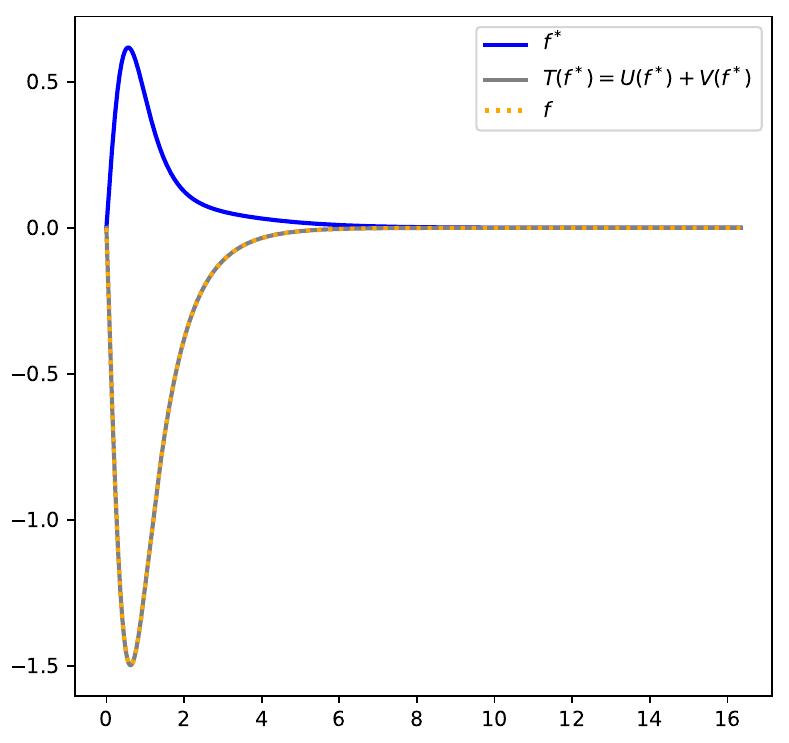}
\caption{The function $f^\ast$ and comparison of $T(f^\ast)$ with $f$ in Lemma \ref{L:fstar}.} \label{F:fstar}
\end{figure}
We next obtain the coefficients $\beta, \gamma_1, \gamma_2$ for the function $f^\ast = \beta f + \gamma_1h_1 + \gamma_2 h_2$:
$$
\beta = -0.374,  \quad \gamma_1 = -0.288,  \quad \gamma_2 = 1.245,
$$
and plot this function as well as comparison of $T(f^\ast)$ with $f$ in Figure \ref{F:fstar}.

Finally, we check the sign of $\lla f,f^\ast\rra$; 
in particular, normalizing both $f$ and $f^\ast$ allows us to check the angle between $f$ and $f^\ast$:
\begin{equation}\label{E:angle}
\cos ~\mbox{of angle }(f,f^\ast) = \frac{\lla f,f^\ast\rra}{\lla f,f \rra^{1/2} \lla f^\ast,f^\ast\rra^{1/2}} = -0.9893
<0,
\end{equation}
which confirms the last item (c) of Corollary \ref{C:main}; thus, concluding that the spectral property for the virial operator holds true in $p=3$ case.

\subsection{2d subcritical ZK case ($1 < p < 3$)}\label{S:R2}
Here, we consider the values of nonlinearity in \eqref{E:Q} $1.1 \leq p \leq 3$ and verify conditions of Corollary \ref{C:main}, which are listed in Table \ref{T:1}. 

{ \footnotesize 
\begin{table}[h!]
\centering
\renewcommand{\arraystretch}{1.2}
\begin{tabular}{|c|c|c|c|c|}
\hline
power {\boldmath$p$}  & \multicolumn{2}{|c|}{Eigenvalues of $B$  
} & {$\det M_p$} & $\cos$ $\mbox{of angle}$  $(f, f^\ast)$\\
\cline{2-3} 
& $\lambda_1$ & $\lambda_2$ & &  
\\
\hline \hline
3.0 & -17.885 & 0.327 & -0.593 & -0.9893 \\
2.9 & -13.019 & 0.250 & -0.345 & -0.9920 \\
2.8 & -9.396 & 0.189 & -0.1967 & -0.9940 \\
2.7 & -6.725 & 0.140 & -0.1094 & -0.9955 \\
2.6 & -4.772 & 0.102 & -0.0592 & -0.9965 \\
2.5 & -3.357 & 0.072 & -0.0310 & -0.9972 \\
2.4 & -2.343 & 0.05  & -0.0157 & -0.9978 \\
2.3 & -1.625 & 0.033  & -0.0076  & -0.9982 \\
2.2 & -1.120 & 0.021   & -0.0035  & -0.9985 \\
2.1 & -0.771 & 0.013  & -0.0015  & -0.9988 \\
2.0 & -0.530 & 0.007 & -0.000613  & -0.9990 \\
1.9 & -0.366 & 0.004  & -0.000227  & -0.9992 \\
1.8 & -0.255 & 0.002  & -7.53 $\times 10^{-5}$  & -0.99936 \\
1.7 & -0.179 & 0.001  & -2.18 $\times10^{-5}$  & -0.99951 \\
1.6 & -0.126  & 0.00025  & -5.25 $\times10^{-6}$  & -0.99965 \\
1.5 & -8.94 $\times 10^{-2}$ & 6.34 $\times10^{-5}$  & -9.84 $\times 10^{-7}$   & -0.99976 \\
1.4 & -6.21 $\times 10^{-2}$ & 1.25 $\times 10^{-5}$ & -1.28 $\times 10^{-7}$  & -0.999853 \\
1.3 & -4.13 $\times 10^{-2}$  & 1.53 $\times 10^{-6}$  & -9.37 $\times 10^{-9}$   & -0.999921 \\
1.2 & -2.47 $\times 10^{-2}$ & 8.10 $\times 10^{-8}$ & -2.38 $\times 10^{-10}$  & -0.999967 \\
1.1 & -1.12 $\times 10^{-2}$ & 5.65  $\times 10^{-10}$ & -4.55 $\times 10^{-13} $ & -0.9999923 \\
\hline
\end{tabular}
\caption{\small {Nonlinearity power $p$ (decreasing from 3 down to 1.1), eigenvalues of $B$ from \eqref{E:B-matrix}, determinant of the matrix $M$ in \eqref{E:M}, and the normalized inner product or
cosine of the angle between $f$ and $f^\ast$ (to verify the sign) from \eqref{E:angle}.}}
\label{T:1}
\end{table}
}

\vspace{-1cm}
\begin{remark}
While the precision of the determinant of the matrix $M_p$ for $p=1.1$ looks very small, nearly machine precision ($- 4.55 \times 10^{-13}$), this is a non-zero value, as the 3 $\times$ 3 determinant is computed by multiplying 3 values of the order $10^{-5}$, namely, 
$$
M_{1.1} = 
\begin{bmatrix}
-1.234 \times 10^{-1} & -1.528 \times 10^{-5} & -9.741 \times 10^{-5} \\
-2.977 \times 10^{-6} &  4.330 \times 10^{-6} & 2.753 \times 10^{-5} \\
-1.888 \times 10^{-5} & 2.753 \times 10^{-5} & 1.759 \times 10^{-4}
\end{bmatrix}.
$$
\end{remark}

Observe that the values in Table \ref{T:1} confirm that one of the eigenvalues of the matrix $B$ (as in \eqref{E:B-matrix}) is positive and another one is negative, which we also plot for a visualization in Figure \ref{F:eig-B}, thus, confirming the $\ind V = 1$ for all $1.1 \leq p \leq 3$. 

\begin{figure}[ht!]
\includegraphics[width=0.49\hsize,height=0.35\hsize]{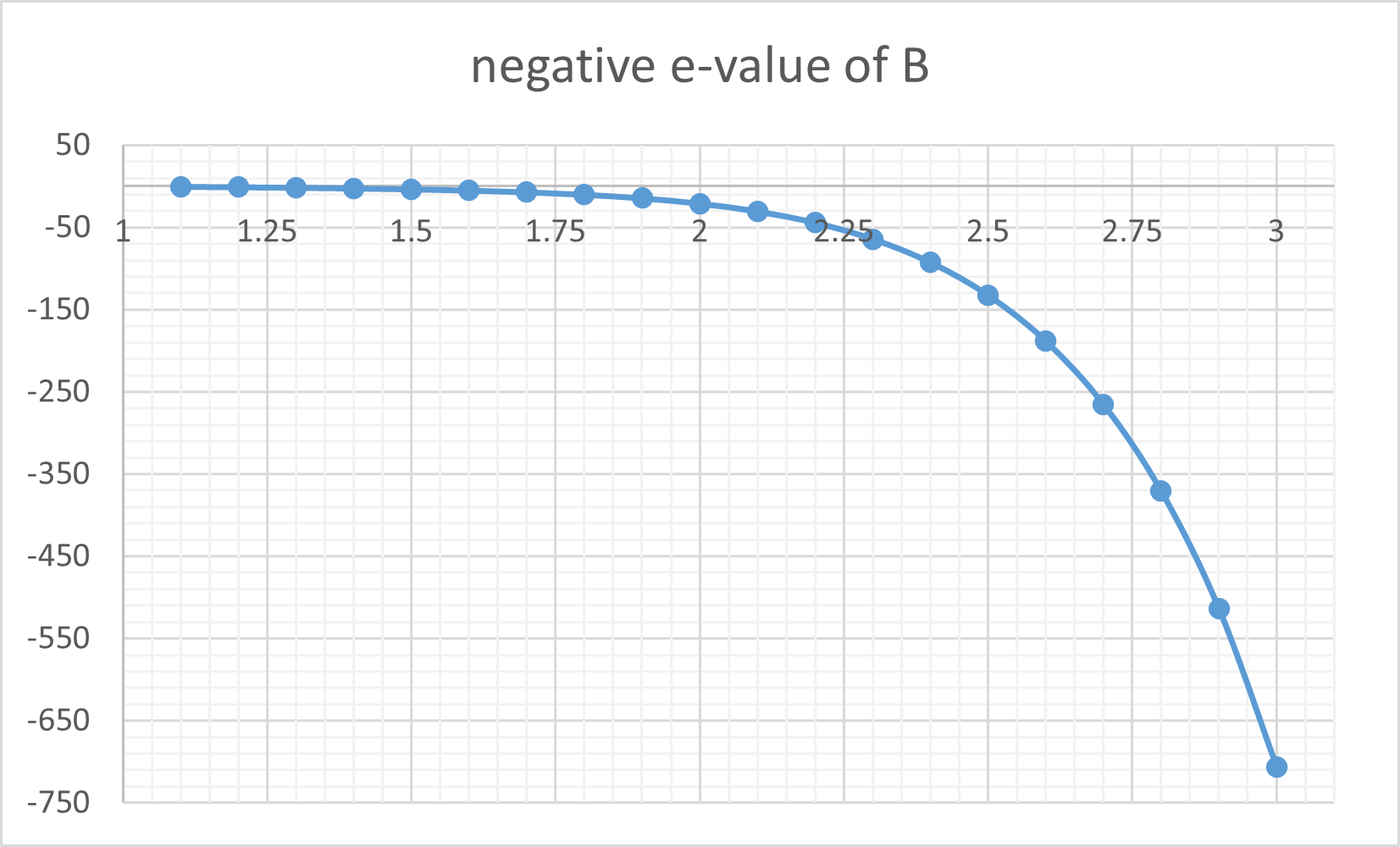}
\includegraphics[width=0.49\hsize,height=0.35\hsize]{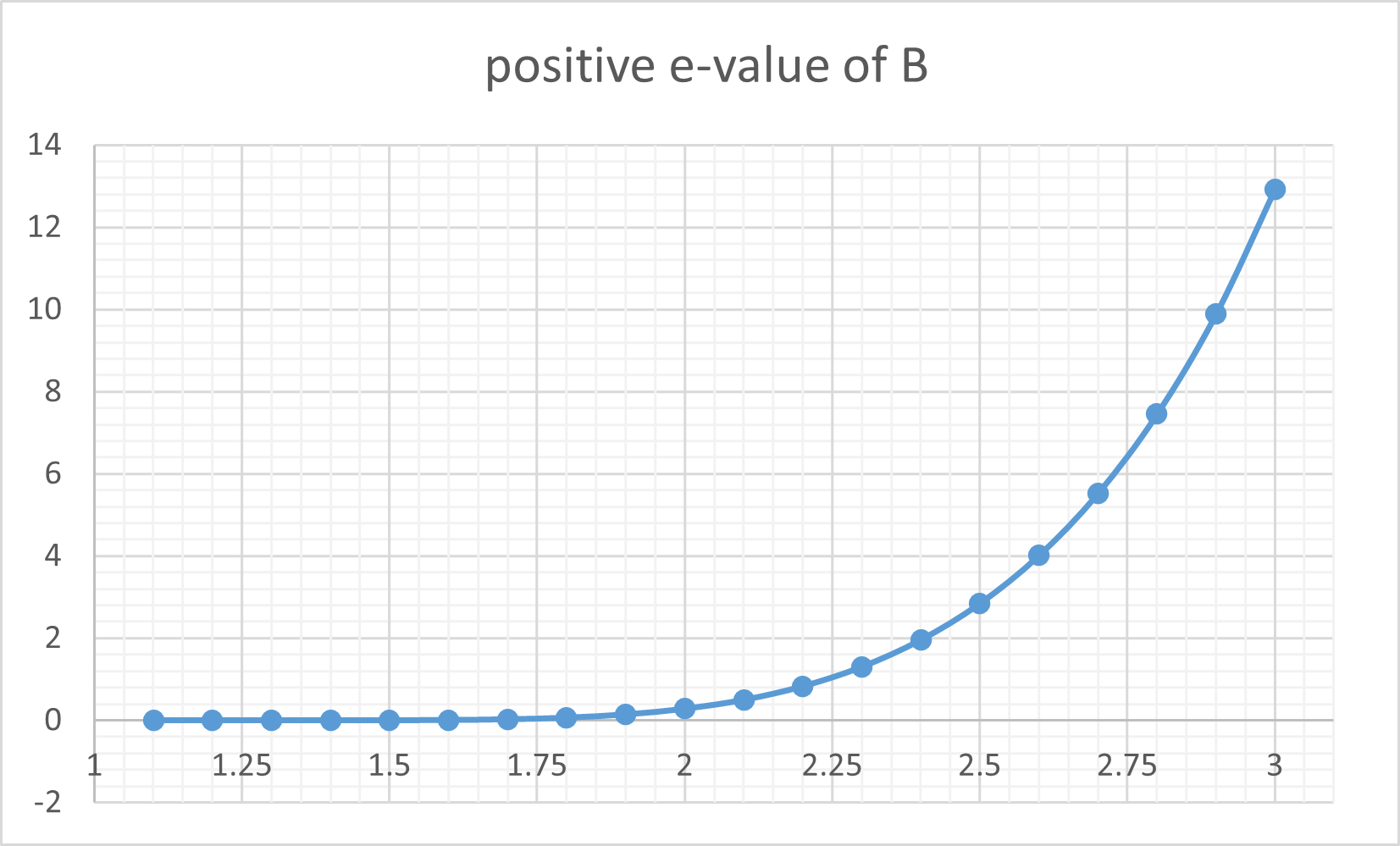}
\caption{\small Eigenvalues of matrix B from Table \ref{T:1}: 
negative (left), positive (right) for $1.1 \leq p\leq 3$.} 
\label{F:eig-B}
\end{figure}

In Figure \ref{F:angle} we show the graph of the determinant $M_p$'s values from Table \ref{T:1} and how the normalized inner product of $f$ and $f^\ast$ (or $\cos$ of the angle between $f$ and $f^\ast$) changes for the same range of $p$. 

\begin{figure}[ht!]
\includegraphics[width=0.49\hsize,height=0.35\hsize]{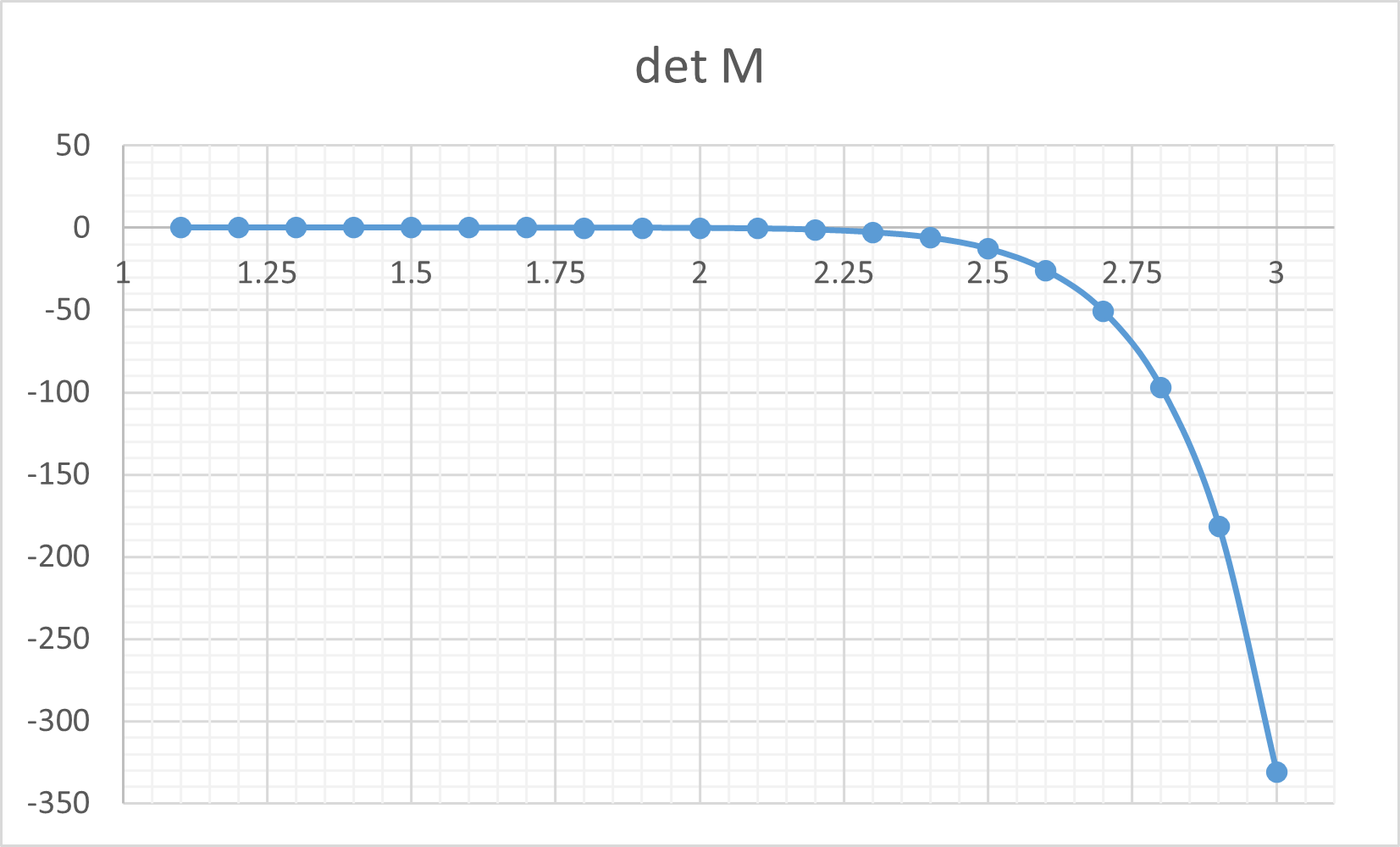}
\includegraphics[width=0.49\hsize,height=0.35\hsize]{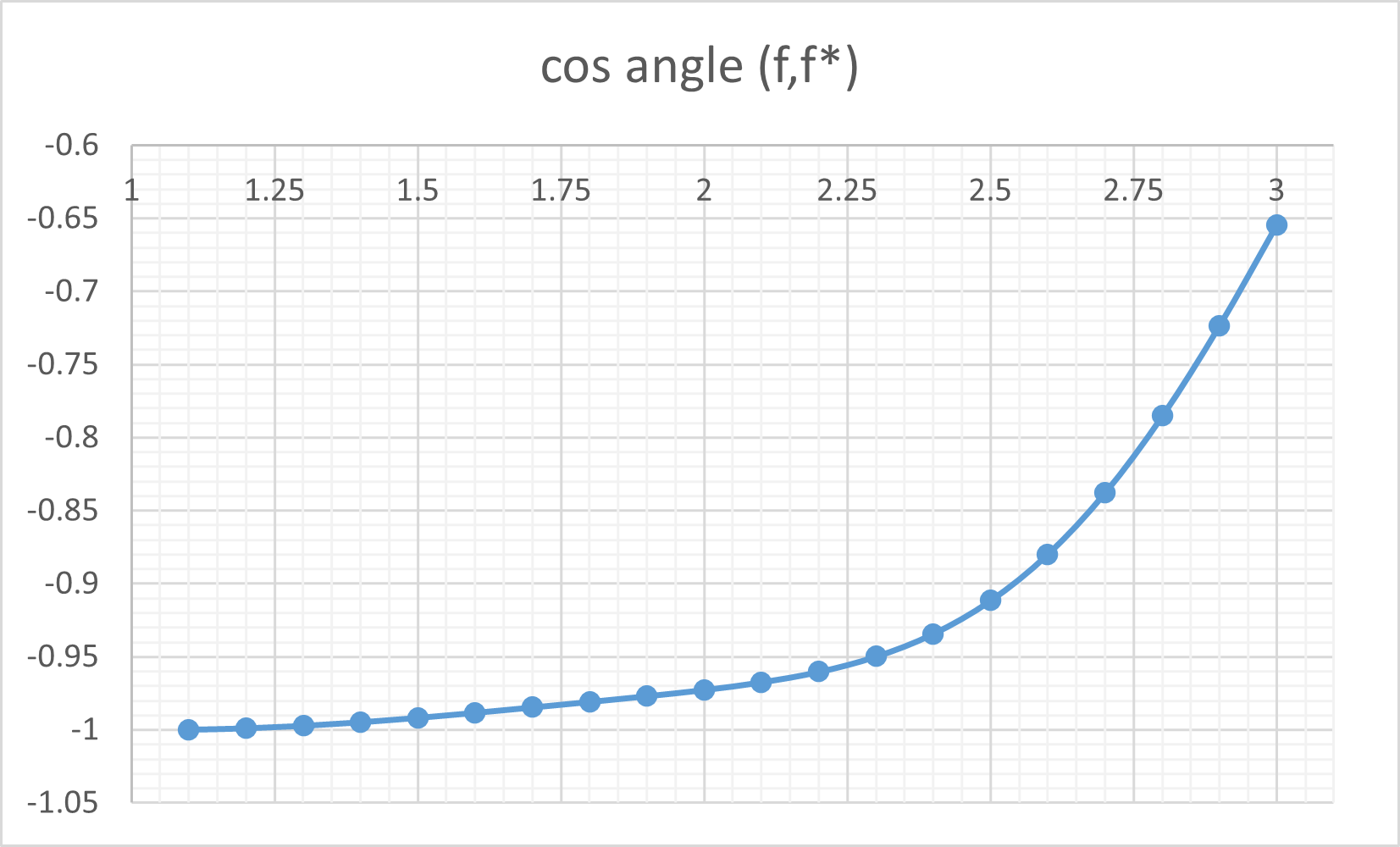}
\caption{\small det $M_p$ and  cosine of the angle between $f$ and $f^\ast$ from Table \ref{T:1}
for $1.1 \leq p\leq 3$.} 
\label{F:angle}
\end{figure}

\bigskip

\bigskip

Note that we have verified that the spectral property of the virial operator holds true for any subcritical and critical cases of $1.1 \leq p \leq 3$. 
For values of $p$ as it approaches 1 ($p \to 1$), we investigate this problem elsewhere; while we believe the results will be similar, it requires more careful numerical investigation as well as the computational power. 
We conclude with mentioning that repeating our approach from \cite{FHRY} and \cite{FHRY2} completes the asymptotic stability of solitary waves in the range of $1.1 \leq p <3$; the conjecture is that the asymptotic stability holds for the entire subcritical range $1<p <3$.

\bigskip

{\bf Conflict of Interest:} The authors declare that they have no conflicts of interest.

\bibliographystyle{abbrv}

\end{document}